\documentclass[article, 11pt]{amsart}

\usepackage{amsmath, amssymb, amsfonts, amsthm, latexsym}
\usepackage{tikz}

\newtheorem{Theorem}{Theorem}[section]
\newtheorem{Lemma}[Theorem]{Lemma}
\newtheorem{Corollary}[Theorem]{Corollary}
\newtheorem{Proposition}[Theorem]{Proposition}
\newtheorem{Remark}[Theorem]{Remark}
\newtheorem{Example}[Theorem]{Example}
\newtheorem{Definition}[Theorem]{Definition}

\newtheorem{Conjecture}[Theorem]{Conjecture}

\def\dim{\operatorname{dim}}
\def\rank{\operatorname{rank}}

\def\degree{\operatorname{deg}}

\def\initial{\operatorname{in}}

\def\lim{\operatorname{lim}}

\newcommand{\pmt}[1]{\begin{pmatrix}#1\end{pmatrix}}

\newcommand{\mt}[1]{\begin{matrix}#1\end{matrix}}

\begin{document}

\title{The initial ideal of generic sequences \\  and Fr\"{o}berg's Conjecture}

\author{Van Duc Trung}
\address{Department of Mathematics, University of Genoa, Via Dodecaneso 35, 16146 Genoa, Italy}
\email{vanductrung@dima.unige.it}


\begin{abstract}
Let $K$ be an infinite field and let $I = (f_1,\cdots,f_r)$ be an ideal in the polynomial ring $R = K[x_1,\cdots,x_n]$ generated by generic   forms  of degrees $d_1,\cdots,d_r$. A longstanding conjecture by Fr\"{o}berg  predicts the shape of the Hilbert function of $R/I.$ In 2010 Pardue  stated  a conjecture on the initial ideal of   $n$ generic forms with respect to the deg-revlex order and he proved that it is equivalent to Fr\"{o}berg's Conjecture.   We study  Pardue's Conjecture and we prove it under  suitable conditions on the degrees of the forms.  This yields a partial solution  to Fr\"{o}berg's Conjecture in the case $r \leq n+2$  over an infinite field of any characteristic.
\end{abstract}

\maketitle


\section{Introduction}
Let $R = K[x_1, \cdots, x_n]$ be the polynomial ring in $n$ variables over an infinite field $K. $   A homogeneous ideal $I $ in $R$ is said to be of type $(n; d_1, \cdots, d_r)$ if  it is  generated by $r$ forms of  degree $d_i $   for $i = 1, \cdots, r. $ The Hilbert function  of $A= R/I$ is by definition $HF_{A}(t) := \dim_K R_t/I_t $ for every $t \ge 0$  and we  are interested in the behavior of the Hilbert function  of  generic   ideals of type  $(n; d_1, \cdots, d_r). $   We adopt  the definition of generic ideals given by Fr\"{o}berg in \cite{F} because it is more suitable for our approach. Assume that   $K$ is an extension of a base field $F$.

\begin{Definition}   \emph{ A form of degree $d$ in $K[x_1, \cdots, x_n]$ is called {\it{generic}}  over $F$ if it is a linear combination of all monomials of degree $d$ and all coefficients are algebraically independent over $F$.
A  homogeneous ideal $(f_1, \cdots, f_r)$ is called generic if all $f_i$ are generic forms with all the coefficients algebraically independent over $F$.}
\end{Definition}

Other definitions  have been given in terms of the affine space parametrized by the coefficients of  the $r$ forms and one says that a property of such sequences is generic if it holds in  a nonempty Zariski open subset of the affine space, see for instance \cite{MS} or \cite{P}. The property ought to hold for a randomly chosen sequence. The Hilbert function of a generic ideal of type $(n;d_1,\cdots,d_n)$ is the Hilbert function of a regular sequence, hence the generating    Hilbert series $HS_A(z) := \sum_{t\ge 0} HF_A(t) z^t  $  is well known and
$HS_A(z) =  \dfrac{\prod_{i=1}^n(1 - z^{d_i})}{(1-z)^n}. $
\vskip0.2cm
For a power series $\sum a_i z^i$ we denote  $\lceil \sum a_i z^i \rceil := \sum b_i z^i$, with $b_i = a_i$ if $a_j > 0$ for all $j \leq i$, and $b_i = 0$ otherwise.
In  1985  Fr\"{o}berg stated the following conjecture:

\begin{Conjecture}\label{Froberg}(Fr\"{o}berg's Conjecture). Let $I = (f_1, \cdots, f_r)$ be a generic homogeneous ideal of type $(n;d_1,\cdots,d_r)$ in $R = K [x_1,\cdots,x_n]$. Then the Hilbert series of $A = R/I$ is given by
$$HS_A(z) = \left\lceil \dfrac{\prod_{i=1}^r(1 - z^{d_i})}{(1-z)^n} \right\rceil $$
\end{Conjecture}

This problem is of central interest in commutative algebra in the last decades   and a great  deal was done (see for instance Anick \cite{A},  Fr\"{o}berg \cite{F}, Fr\"{o}berg-Hollman \cite{FH}, Fr\"{o}berg-L\"{o}fwall \cite{FLo}, Fr\"{o}berg-Lundqvist \cite{FLu}, Moreno-Soc\'{\i}as \cite{MS}, Pardue \cite{P}, Stanley \cite{S}, Valla \cite{V}).  A  large number of validations through  computational methods suggests  a positive answer.  Fr\"{o}berg's Conjecture is clearly  true if $r \leq n$ (complete intersections); it is known if $n \leq 2$ \cite{F}; $n=3$ \cite{A}; $r  = n+1$ (almost complete intersection) with  char$K=0$ \cite{S1}; and some further special cases when all $d_i$ are equal (see \cite{BO}, \cite{FH}, \cite{N}). Our study will contribute to give a new partial solution to Fr\"{o}berg's Conjecture in the cases $r = n+1, n+2$ for any characteristic under a suitable condition on $d_1, \cdots, d_r $    (see Theorem \ref{Partial FB}).   Actually our main goal is to prove an equivalent conjecture stated by Pardue which will be presented below.

Denote by $\initial_{\tau}(I)$  the initial ideal of $I$ with respect to a term order $\tau$ on $R$.  Because the Hilbert series of $R/I$ and of $R/\initial_{\tau}(I) $ coincide  for every $\tau$, a rich literature has been developed with the  aim to characterize  the initial ideal of generic ideals with respect to suitable term orders (see \cite{AJL}, \cite{A}, \cite{JS}, \cite{CP}, \cite{C}, \cite{HW}, \cite{MS1}, \cite{MS}, \cite{P}). From now on,  the initial ideal of $I$  will be  always with respect to the {\it{degree reverse lexicographic order}}  and it will be denoted simply by $\initial(I).$  For general facts and properties on the degree reverse lexicographic order see for instance \cite[Proposition 15.2]{E}.      It is natural to guess that generic complete intersections share {\it{special}}  initial ideals.      We present here Pardue's Conjecture.

 For every monomial $x^{\alpha} \in K[x_1, \cdots , x_n]$, denote by $\max(x^{\alpha})$ the largest index $i$ such that $x_i$ divides $x^{\alpha}$. Pardue stated a conjecture on the initial ideal of a generic homogeneous ideal of type $(n;d_1,\cdots,d_n)$ in $R = K [x_1,\cdots,x_n]$ which is equivalent to Fr\"{o}berg's Conjecture (\cite[Theorem 2]{P}).
\begin{Conjecture}\label{Pardue E} \cite[Conjecture E]{P} Let $I = (f_1,\cdots,f_n)$ be a generic homogeneous ideal of type $(n;d_1,\cdots,d_n)$ in $K [x_1,\cdots,x_n]$. If $x^\mu$ is a minimal generator of $\initial(I)$ with $\max(x^\mu) = m$ and $\deg(x^\mu)=d$, then every monomial of degree $d$ in the variables $x_1,\cdots,x_{m-1}$ must be in $\initial(I)$ as well.
\end{Conjecture}

Actually in  \cite{MS1}  Moreno-Soc\'{\i}as  stated   a stronger conjecture announcing that $\initial(I)$ should be almost reverse lexicographic, i.e, if $x^\mu$ is a minimal generator of $\initial(I)$ then every monomial of the same degree and greater than $x^\mu$ must be in $\initial(I)$ as well. Moreno-Soc\'{\i}as' Conjecture was proven  in the case $n=2$ by Aguire et al. \cite{AJL} and Moreno-Soc\'{\i}as \cite{MS}, $n=3$ by Cimpoeas \cite{C}, $n=4$ by Harima and Wachi \cite{HW} and for certain sequences $d_1, \cdots, d_n$   by Cho and Park assuming char$K=0 $ \cite{CP}. Without restriction on the characteristic of $K, $ by using an incremental method introduced in  \cite{GGV}, Capaverde and Gao improved the result of Cho and Park, see \cite[Theorem 3.19]{JS}.

Inspired by the incremental method by Capaverde and Gao, in Proposition \ref{Structure initial}, we give an explicit  description of  the initial ideal of generic ideals with respect to the degree reverse lexicographic order. From this description, we obtained a partial solution to Conjecture \ref{Pardue E} under  suitable conditions on the degree of the forms. We hope that this approach will be successfully applied to give new insights in proving Pardue's Conjecture and hence Fr\"{o}berg's Conjecture.
In this paper we give a proof of Conjecture \ref{Pardue E} in case $n \leq 3, $ see Theorem \ref{Partial Pardue E3}. If $n \geq 4$, we prove Conjecture \ref{Pardue E} under the following condition on $d_1,\cdots,d_n$.  Let $d_1 \leq \cdots \leq d_n$ and for every $1 \leq i \leq n$, we set
\vskip 0.3cm
\hspace{120pt} $\delta_i = d_1 + \cdots + d_{i} - i,$
\vskip 0.3cm
\hspace{120pt} $\sigma_i = \min \left\{ \delta_{i-1} , \left\lfloor \dfrac{\delta_i}{2} \right\rfloor \right\}$ for all $i \geq 2$.
\vskip 0.3cm
In Theorem \ref{Partial Pardue E}, we prove the following crucial result which requires a technical part on the structure of the initial ideal of a generic complete intersection.
\vskip 0.2cm
\noindent \textbf{Theorem 1.} \emph{Let $I = (f_1,\cdots,f_n)$ be a generic homogeneous ideal of type $(n;d_1,\cdots,d_n)$ in $K [x_1,\cdots,x_n]$ with $n \geq 4$ and $d_1 \leq \cdots \leq d_n$. If $d_i \geq \sigma_{i-1}$ for all $4 \leq i \leq n$, then Conjecture \ref{Pardue E} is true for $I$.}
 \vskip 0.2cm
As  application,  in Theorem \ref{Partial FB}, we prove the following result.
\vskip 0.3cm
\noindent \textbf{Theorem 2.} \emph{Let $I = (f_1,\cdots,f_r)$ be a generic homogeneous ideal of type $(n;d_1,\cdots,d_r)$ in $R = K [x_1,\cdots,x_n]$ with $r \leq n+2$ and $d_1 \leq \cdots \leq d_r$. If $r \leq 3$ or $r \geq 4$ and $d_i \geq \sigma_{i-1}$ for every $4 \leq i \leq r$, then the Hilbert series of $R/I$ is given by}
$$HS_{R/I}(z) =  \left\lceil \dfrac{\prod_{i=1}^r(1 - z^{d_i})}{(1-z)^n} \right\rceil.$$

All the computations in this paper have been performed by using Macaulay2 \cite{GS}.


\section{Preliminaries}
Let $R' = K[x_1,\cdots,x_n,z]$ be the polynomial ring in $n+1$ variables and fix the order on the variables $x_1 > \cdots > x_n > z$. Let $f_1,\cdots,f_n$ and $g$ be generic forms in $R'$ such that $I=(f_1,\cdots,f_n)$ is a generic homogeneous ideal of type $(n+1;d_1,\cdots,d_n)$ and $(I,g) = (f_1,\cdots,f_n,g)$ is a homogeneous ideal of type $(n+1;d_1,\cdots,d_n,d)$. Define $\pi: R' \longrightarrow R = K[x_1,\cdots,x_n]$ to be the ring homomorphism where  $z $ goes to zero, fixing the elements in $K$ and the variables $x_1,\cdots,x_n$. Let $J = \pi(I)$ be the image of $I$. Then $J$ is a generic homogeneous ideal of type $(n;d_1,\cdots,d_n)$ in $R$. Recall that we always consider  the degree reverse lexicographic order.
\begin{Proposition}
$\initial(I)$ and \ $\initial(J)$ have the same minimal generators.
\end{Proposition}

\begin{proof}
From a property of the  degree reverse lexicographic order in \cite[Proposition 15.12]{E}, we get $\pi(\initial(I)) = \initial(J)$. On the other hand, since $z$ is regular in $R'/I$, by \cite[Theorem 15.13]{E} $z$ is regular in $R'/\initial(I)$. Furthermore, by \cite[Theorem 15.14]{E} the minimal generators of $\initial(I)$ are not divisible by $z$. Thus, $\initial(I)$ and \ $\initial(J)$ have the same minimal generators.
\end{proof}
Let $B = B(J)$, which is called the set of standard monomials with respect to $J$, be the set of monomials in $R$ that are not in $\initial(J)$. We set
\vskip 0.3cm
\hspace{120pt} $\delta = \delta_n = d_1 + \cdots + d_n - n,$
\vskip 0.3cm
\hspace{120pt} $\sigma = \sigma_n = \min \left\{ \delta_{n-1} , \left\lfloor \dfrac{\delta_n}{2} \right\rfloor \right\}.$
\vskip 0.3cm
It is known that $A = R/J$ is an Artinian complete intersection and the Hilbert series of $A$, say $HS_A(z) = \sum_{i = 0}^{\delta} a_i z^i,$ is a symmetric polynomial of degree $\delta$ (i.e., $a_i = a_{\delta-i}, \forall  i$), with
$$0 < a_0 < a_1 < \cdots < a_{\sigma} = \cdots = a_{\delta - \sigma} > \cdots > a_{\delta - 1} > a_\delta > 0,$$
(see for instance   \cite[Proposition 2.2]{MS}). Notice that  $a_i = | B_i |$ where $B_i$ is the set of monomials of degree $i$ in $B$.
\vskip 0.3cm
The set of standard monomials with respect to a generic ideal of type $(n;d_1,\cdots,d_r)$ depends only on $(n; d_1, \cdots, d_r)$ and we will denote it by $B(n;d_1, \cdots, d_r)$. We will describe more clearly the set of standard monomials in the case $r=n, $  that is $B=B(n;d_1, \cdots, d_n).$ For each $ 1 \leq i \leq \sigma$, define
$$\widetilde{B}_i = \{ x^\alpha \in B_i  \ | \ \max(x^\alpha) < n \}.$$
\vskip 0.3cm
\begin{Proposition} \label{Structure B} The structure of $B = B(n;d_1, \cdots, d_n)$ is as follows,
\vskip 0.3cm
{\rm (1)} $B_i = \widetilde{B}_i  \cup  x_n \widetilde{B}_{i-1} \cup \cdots \cup x_n^{i-1}\widetilde{B}_1 \cup \{ x_n^i \}$
\vskip 0.3cm
\hspace*{30pt} $= \widetilde{B}_i \cup x_nB_{i-1},$ for all $1 \leq i \leq \sigma$.
\vskip 0.3cm
{\rm (2)} $B_{\sigma+i} = x_n^i \widetilde{B}_{\sigma}  \cup  x_n^{i+1} \widetilde{B}_{\sigma-1} \cup \cdots \cup x_n^{\sigma +i-1}\widetilde{B}_1 \cup \{ x_n^{\sigma+i} \}$
\vskip 0.3cm
\hspace*{42pt} $= x_n^iB_{\sigma}$, for all $0 \leq i \leq \delta - 2\sigma$.
\vskip 0.3cm
{\rm (3)} $B_{\delta-i} = x_n^{\delta - 2i} \widetilde{B}_{i}  \cup  x_n^{\delta -2i+1} \widetilde{B}_{i-1} \cup \cdots \cup x_n^{\delta -i-1}\widetilde{B}_1 \cup \{ x_n^{\delta-i} \}$
\vskip 0.3cm
\hspace*{42pt} $= x_n^{\delta - 2i}B_i$, for all $ 0 \leq i \leq \sigma$.
\end{Proposition}

\begin{proof}
(1) For $1 \leq i \leq \sigma$, we have $a_{i-1} < a_i$. Let $S$ denote the subset of $B_i$ consisting of the $a_{i-1}$ smallest monomials in $B_i$ with respect to the degree reverse lexicographic order . By \cite[Lemma 3.5]{JS} we get $S = x_nB_{i-1}$. It is clear that $\widetilde{B}_i \subseteq B_i \setminus S$. Conversely,  assume that for some  monomial $x^\alpha \in B_i \setminus S  $  and $x^\alpha \notin \widetilde{B}_i$. Then $x_n$ divides $x^\alpha$, so that $x^\alpha / x_n \in B_{i-1}$. This implies a contradiction since $x^\alpha  = x_n(x^\alpha / x_n) \in S$. Thus, $\widetilde{B}_i = B_i \setminus S$ and (1) holds.
\vskip 0.3cm
\noindent (2) For $0 \leq i \leq \delta - 2\sigma$, we have $a_{\sigma+i} = a_\sigma$. By \cite[Lemma 3.5]{JS} we get $B_{\sigma+i} = x_n^iB_{\sigma}$.
\vskip 0.3cm
\noindent (3) Since $\sigma \leq \dfrac{\delta}{2}$, by \cite[Lemma 3.4]{JS} we get $B_{\delta-i} = x_n^{\delta - 2i}B_i$  for $0 \leq i \leq \sigma$.
\end{proof}

\begin{Remark} \emph{(1) $B=B(n;d_1,\cdots,d_n)$  is determined by $\widetilde{B}_1,  \cdots, \widetilde{B}_{\sigma}$.
\vskip 0.3cm
\noindent (2) $|\widetilde{B}_i| = a_i - a_{i-1} = a'_i > 0$, for all  $1 \leq i \leq \sigma$.
\vskip 0.3cm
\noindent (3) For $\sigma < i \leq \delta$ the monomials in $B_i$ are divisible by $x_n$.}
\end{Remark}
In the following example, we describe  explicitly $B(n; d_1, \cdots, d_n)$ according to Proposition \ref{Structure B}.
\begin{Example} \label{ex structure B}
\emph{Let $B = B(4;2,3,3,4)$ be the set of standard monomials with respect to a generic ideal of type $(4;2,3,3,4)$. Then $\delta = 8; \sigma=4$. Denote by $a_i = |B_i|$ and $a'_i = |\widetilde{B}_i|$. We have
\begin{align*}
&B_0 = \{ 1 \}  & a_0 = 1.\\
&B_1 = \widetilde{B}_1 \cup \{ x_4 \} \ \text{where} \ \widetilde{B}_1 = \{ x_1, x_2, x_3 \} &  a'_1 = 3, a_1 = 4.\\
&B_2 = \widetilde{B}_2 \cup x_4B_1 \ \text{where} \ \widetilde{B}_2 = \{ x_1x_2, x_2^2, x_1x_3, x_2x_3, x_3^2 \} & a'_2 = 5, a_2 = 9.\\
&B_3 = \widetilde{B}_3 \cup x_4B_2 \ \text{where} \ \widetilde{B}_3 = \{ x_1x_2x_3, x_2^2x_3, x_1x_3^2, x_2x_3^2, x_3^3 \} &  a'_3 = 5, a_3 = 14.\\
&B_4 = \widetilde{B}_4 \cup x_4B_3 \ \text{where} \ \widetilde{B}_4 = \{ x_2x_3^3,  x_3^4 \} &  a'_4 = 2, a_4 = 16.\\
&B_5 = x_4^2B_3 &  a_5 = a_3 = 14.\\
&B_6 = x_4^4B_2 &  a_6 = a_2 = 9.\\
&B_7 = x_4^6B_1 & a_7 = a_1 = 4.\\
&B_8 = x_4^8B_0 &  a_8 = a_0 = 1.
\end{align*}}
\end{Example}
In the next section, we will use the above example   to construct $B(n+1, d_1, \cdots, d_{n+1})$ starting from $B(n, d_1, \cdots, d_n)$ by using the incremental method introduced in \cite{GGV} and adapted to our situation in \cite{JS}.

\section{Main results}
Let $(I,g) = (f_1,\cdots,f_n,g)$ be a generic homogeneous ideal of type $(n+1;d_1,\cdots,d_n,d)$ in $R' = K[x_1,\cdots,x_n,z]$. Define $C_I$ to be the set of the coefficients of the polynomials $f_1,\cdots,f_n$ and $\bar{F} = F(C_I) \subset K$, where $K$ is an extension of a base field $F$. Let $G = \{ g_1, \cdots,g_t \}$ be the reduced Gr\"{o}bner basis of $I$ with respect to the degree reverse lexicographic order.  Then $g_1, \cdots, g_t$ are homogeneous polynomials in $\bar{F}[x_1,\cdots,x_n,z]$.
\vskip 0.2cm
Let $E = B(n+1, d_1, \cdots, d_n)$ be the set of  standard monomials with respect to $I$. Reducing $g$ modulo $G$ we obtain a polynomial which is a $K$-linear combination of all monomials of degree $d$ in $E$ with coefficients still algebraically independent over $\bar{F}$. Hence, from now on we assume that $g$ is reduced modulo $G$ and the coefficients of $g$ are algebraically independent over $\bar{F}$.
\vskip 0.2cm
In order to construct $B(n+1,d_1,\cdots,d_{n+1})$ from $B(n,d_1,\cdots,d_n)$ we need to compare $\initial(I,g)$ and $\initial(I)$. We recall here the incremental method to construct $\initial(I,g)$ from $\initial(I)$; for more details see \cite[Section 3]{JS}.  Let $B = B(n; d_1, \cdots, d_n)$ be the set of  standard monomials with respect to $J = \pi(I)$. For every $i \geq 0$, denote by $E_i$ the set of monomials of degree $i$ in $E$.  Note that, for $0 \leq i \leq \delta$, we have
$$E_i = B_i \cup zB_{i-1} \cup z^2B_{i-2} \cup \cdots \cup z^{i-1}B_1 \cup z^iB_0,$$
\noindent and for $i > \delta$, we have $E_i = z^{i - \delta}E_\delta.$
\vskip 0.2cm
Let $0 \leq i \leq \delta$, denote by $\mathbf{E}_i$  the column vector whose entries are the monomials in $E_i$ listed in decreasing order with respect to the degree reverse lexicographic order . For each monomial $x^{\alpha} \in \mathbf{E}_i$, reducing the product $x^{\alpha}g \in R'_{i+d}$ modulo $G$ we obtain a polynomial, say the reduced form of $x^{\alpha}g$, that is a $K$-linear combination of monomials in $E_{i+d}$. Note that each coefficient of the reduced form of $x^{\alpha}g$ is a $\bar{F}$-linear combination of coefficients of polynomial $g$.  Let $M_i$ denote the matrix such that
\begin{equation}\label{Equation}
\mathbf{E}_i.g \equiv M_i\mathbf{E}_{i+d} \ \text{(mod} \ G),
\end{equation}
where $\mathbf{E}_{i+d}$ denotes the column vector whose entries are the monomials in $E_{i+d}$ listed in decreasing order with respect to the degree reverse lexicographic order. Thus, each entry of matrix $M_i$ is a $\bar{F}$-linear combination of coefficients of polynomial $g$. By \cite[Lemma 3.2]{JS} the rows of $M_i$ are linearly independent. This means that $\rank(M_i) = |E_i| = a_i + a_{i-1} + \cdots + a_0$. Furthermore, the monomials in $\mathbf{E}_{i+d}$ corresponding to the $|E_i|$ first linearly independent columns of $M_i$ are the generators that will be added to $\initial(I)$ to form $\initial(I,g)$. Note that some of the monomials we add might be redundant. In this section, we will prove the following  result which  will be fundamental in our approach.

\begin{Theorem}\label{Pardue E with condition}
Let $(I,g) = (f_1,\cdots,f_n,g)$ be a generic homogeneous ideal of type $(n+1;d_1,\cdots,d_n,d)$ in $R' = K [x_1,\cdots,x_n,z]$, where $d_1 \leq \cdots \leq d_n \leq d$. If  $d \geq \sigma$ and Conjecture \ref{Pardue E} is true for $J = \pi(I)$, then Conjecture \ref{Pardue E} is also true for $(I,g)$.
\end{Theorem}

The proof is technical and it needs a deep investigation given in Proposition \ref{Structure initial}, Proposition \ref{Part 2} and Proposition \ref{Part 1}. Let us fix the property stated in   Conjecture \ref{Pardue E}.
\begin{Definition}
\emph{Let $I$ be a homogeneous ideal in $K[x_1,\cdots,x_n]$. Let $x^{\alpha}$ be a monomial in $K[x_1,\cdots,x_n]$ with $\max(x^\alpha) = m$ and $\deg(x^\alpha)=d$. We say $x^\alpha$ satisfies property $P$ with respect to $I$ if every monomial of degree $d$ in the variables $x_1,\cdots,x_{m-1}$ is in the initial ideal of $I$.}
\end{Definition}
Thus, Conjecture \ref{Pardue E} is true for a generic homogeneous ideal $I$ if and only if every minimal generator of $\initial(I)$ satisfies property $P$ with respect to $I$. First of all we notice that Theorem \ref{Pardue E with condition} can be deduced from \cite[Proposition 3.12]{JS} when $d \geq \delta$. Indeed, by \cite[Proposition 3.12]{JS}, if $d \geq \delta$ then
$$\initial(I,g) = (\initial(I), z^{d - \delta}B_{\delta}, z^{d-\delta+2}B_{\delta -1}, \cdots, z^{\delta+d-2}B_1, z^{\delta+d}B_0).$$
Let $x^\mu$ be a generator of $\initial(I,g)$ in $z^{d - \delta}B_{\delta}, z^{d-\delta+2}B_{\delta -1}, \cdots, z^{\delta+d-2}B_1, z^{\delta+d}B_0$. We claim  $x^\mu$ satisfies property $P$ with respect to $(I,g)$. Indeed, if $d > \delta$ then $x^\mu$ is divisible by $z$ and $\degree(x^\mu) = k > \delta$. Hence, every monomial $x^\alpha$ of degree $k$ in variables $x_1, \cdots, x_n$ is not in $E_k = z^{k-\delta}E_\delta$, so that $x^\alpha \in \initial(I) \subset \initial(I,g)$. If $d = \delta$ then
$$\initial(I,g) = (\initial(I),x_n^{\delta}, z^2B_{\delta -1}, \cdots, z^{2\delta-2}B_1, z^{2\delta}B_0).$$
If $x^\mu$ is a monomial in $z^2B_{\delta -1}, \cdots, z^{2\delta-2}B_1, z^{2\delta}B_0$, then $x^\mu$ satisfies property $P$ with respect to $(I,g)$ through an analogous argument as the case $d > \delta$. On the other hand, it is not hard to see that $x_n^\delta$ also satisfies property $P$ with respect to $(I,g)$. Thus, if Conjecture \ref{Pardue E} is true for $J$ with $d \geq \delta$, then every minimal generator of $\initial(I,g)$ satisfies property $P$ with respect to $(I,g)$. This means Conjecture \ref{Pardue E} is true for $(I,g)$.
\vskip 0.3cm
Consider now the case $d < \delta$.  Set $i^*=\lfloor\frac{\delta-d}{2}\rfloor$. The following lemma will be useful for proving Proposition \ref{Structure initial}.
\begin{Lemma}\label{multiple}
For $i > j \geq i^*$, we have $a_{d+i} \leq a_{d+j}$. Furthermore, the monomials of $B_{d+i}$ are multiples of the $a_{d+i}$ smallest monomials in $B_{d+j}$ with respect to the degree reverse lexicographic order.
\end{Lemma}
\begin{proof}
Since $d+i^* \geq  \lfloor\frac{\delta}{2}\rfloor \geq \sigma$ , so  $d+i > d+j \geq \sigma$. Hence, $a_{d+i} \leq a_{d+j}$. Let $S$ denote the subset of $B_{d+j}$ consisting of the $a_{d+i}$ smallest monomials in $B_{d+j}$ with respect to the degree reverse lexicographic order. By \cite[Lemma 3.5 (ii)]{JS}  we have $B_{d+i} = x_n^{i-j}S.$
\end{proof}

By convention, we use the following notation. Let $B$ be a finite subset of monomials in $R = K[x_1,\cdots,x_n]$ and denote by $\mathbf{B} = \{ x^{\alpha_1}, x^{\alpha_2}, \cdots, x^{\alpha_m} \}$ the set of monomials in $B$ listed in decreasing order with respect to the degree reverse lexicographic order. Let $S$ be a subset of $\{1, 2, \cdots, m\}$ and denote $\mathbf{B}^S = \{ x^{\alpha_i} \in B \ | \ i \in S \}$. The set of generators of $\initial(I,g)$ can be described  as the following.
\begin{Proposition}\label{Structure initial}
Let $(I,g) = (f_1,\cdots,f_n,g)$ be a generic ideal of type $(n+1;d_1,\cdots,d_n,d)$ in $R' = K [x_1,\cdots,x_n,z]$, where $d_1 \leq \cdots \leq d_n \leq d$ and $d < \delta$. Let $B = B(n;d_1,\cdots,d_n)$ be the set of  standard monomials with respect to $J = \pi(I)$.
\vskip 0.3cm
\noindent {\rm (1)} If $\delta - d  = 2k$, where $k$ is a positive integer, then
\begin{center}
$\initial(I,g) = (\initial(I), \mathbf{B}_d^{\{1\}}, \mathbf{B}_{d+1}^{S_1}, \mathbf{B}_{d+2}^{S_2}, \cdots, \mathbf{B}_{d+k}^{S_k},  z^2B_{d+k-1}, z^4B_{d+k-2}, \cdots, z^{2(d+k)}B_0),$
\end{center}
{\rm (2)} If $\delta - d  = 2k + 1$, where $k$ is a positive integer, then
\begin{center}
$\initial(I,g) = (\initial(I), \mathbf{B}_d^{\{1\}}, \mathbf{B}_{d+1}^{S_1}, \mathbf{B}_{d+2}^{S_2}, \cdots, \mathbf{B}_{d+k}^{S_{k}}, B_{d+k+1}, zB_{d+k}, z^3B_{d+k-1}, \cdots, z^{2(d+k)+1}B_0),$
\end{center}
{\rm (3)} If $\delta - d  = 1$ then $\initial(I,g) = (\initial(I), \mathbf{B}_d^{\{1\}}, B_{d+1}, zB_d, z^3B_{d-1}, \cdots, z^{2d+1}B_0),$
\vskip 0.2cm
\noindent where  in {\rm (1)} and {\rm(2)}, $S_i$ is a subset of $\{1, 2, \cdots, a_{d+i}\}$ containing $a_i$ elements, for every $i = 1, 2, \cdots , k$.
\end{Proposition}
\begin{proof}
Since $g$ is a combination of all monomials in $E_d = B_d \cup zB_{d-1} \cup \cdots \cup z^dB_0$, we can write
$$g = \mathbf{v}_d\mathbf{B}_d + \mathbf{v}_{d-1}\mathbf{B}_{d-1}z + \cdots + \mathbf{v}_1\mathbf{B}_1z^{d-1} + \mathbf{v}_0z^d,$$
where $\mathbf{v}_i$ is the row vector of the coefficients of $g$ corresponding to the monomials in $\mathbf{B}_i$. Denote the last coefficient of $\mathbf{v_d}$ by $c_d$. Note that $c_d$ is the coefficient corresponding to the monomial $x_n^d$. Set $\mathbf{v}_d^* = \mathbf{v}_d \setminus \{ c_d \}$. We will construct a set of generators for $\initial(I,g)$ by using incremental method. According to equation (\ref{Equation}),  $\mathbf{E}_i.g \equiv M_i\mathbf{E}_{i+d} \ \text{(mod} \ G)$, for each $i$ from $0$ to $\delta$,  we find the monomials that will be added to $\initial(I)$.
\vskip 0.3cm
\noindent $\bullet$ For $i = 0$, we have $E_0 = \{ 1 \}$ and $M_0$ is a row matrix $\pmt{\mathbf{v}_d & \mathbf{v}_{d-1} & \cdots & \mathbf{v}_1 & \mathbf{v}_0}$. Hence, the first column of $M_0$ is linearly independent. Thus, the largest monomial of $\mathbf{B}_d$ will be added to $\initial(I)$.
\vskip 0.3cm
\noindent $\bullet$ For $1 \leq i \leq i^* = \lfloor\frac{\delta-d}{2}\rfloor = k$ (in the case $\delta -d \geq 2$), we have
$$E_i = B_i \cup zB_{i-1} \cup \cdots \cup z^{i-1}B_1 \cup z^iB_0,$$
and
$$E_{d+i} = B_{d+i} \cup zB_{d+i-1} \cup \cdots \cup z^{i-1}B_{d+1} \cup z^iB_d \cup \cdots \cup z^{d+i}B_0.$$
\vskip 0.3cm
Therefore equation (\ref{Equation}) can be explicitly written in the following form

\begin{equation}\label{Equation 1}
\mt{& \mt{\mathbf{B}_{d+i} & z\mathbf{B}_{d+i-1} &  \ \  \cdots & z^{i-1}\mathbf{B}_{d+1} & z^i\mathbf{B}_d & \ \cdots &  z^{d+i}} \\
\pmt{\mathbf{B}_i\\ z\mathbf{B}_{i-1}\\ \vdots \\ z^{i-1}\mathbf{B}_1\\ z^i\mathbf{B}_0}.g  \equiv & \pmt{\ \Gamma_{i,d+i}&\Gamma_{i,d+i-1}&\cdots&\Gamma_{i,d+1}&\Gamma_{i,d}&\ \cdots&\Gamma_{i,0}\\
\ 0&\Gamma_{i-1,d+i-1}&\cdots&\Gamma_{i-1,d+1}&\Gamma_{i-1,d}&\ \cdots&\Gamma_{i-1,0}\\
\ \vdots&\vdots&\ddots&\vdots&\vdots&\ \ddots&\vdots\\
\ 0&0&\cdots&\Gamma_{1,d+1}&\Gamma_{1,d}&\ \cdots& \Gamma_{1,0}\\
\ 0&0&\cdots&0&\Gamma_{0,d}&\ \cdots&\Gamma_{0,0}}},
\end{equation}
\vskip 0.3cm
\noindent where the entries of block $\Gamma_{j,l}$, for $0 \leq j \leq i$ and $0 \leq l \leq d+i$, are the coefficients corresponding to the monomials in $z^{d+i-l}B_l$, in the reduced form of the polynomials in $z^{i-j}gB_j$.
\vskip 0.3cm
Denote by $A_i, A_{i-1}, \cdots, A_1, A_0$ the submatrices of $M_i$ formed by the columns corresponding to the monomials in $\mathbf{B}_{d+i}, z\mathbf{B}_{d+i-1}, \cdots, z^{i-1}\mathbf{B_{d+1}}, z^i\mathbf{B_d}$ respectively.
\vskip 0.3cm
We now consider the block $\Gamma_{i,d+i}$. Note that the entries of $\Gamma_{i,d+i}$ are the $\bar{F}$-linear combinations of the coefficients in $\mathbf{v}_d$. Since $ i \leq \dfrac{\delta-d}{2}$, we have $i < \sigma$ and
$$\mathbf{B}_i = \widetilde{\mathbf{B}}_i  \cup  x_n \widetilde{\mathbf{B}}_{i-1} \cup \cdots \cup x_n^{i-1}\widetilde{\mathbf{B}}_1 \cup \{ x_n^i \}.$$
Since $d+i \leq \delta -i$, we have $a_{d+i} \geq a_i$ and the $a_i$ smallest monomials in $\mathbf{B}_{d+i}$ are
$$x_n^d\mathbf{B}_i = x_n^d\widetilde{\mathbf{B}}_i  \cup  x_n^{d+1} \widetilde{\mathbf{B}}_{i-1} \cup \cdots \cup x_n^{d+i-1}\widetilde{\mathbf{B}}_1 \cup \{ x_n^{d+i} \}.$$
Let $x^{\alpha} \in \mathbf{B}_i$. Suppose $x^{\alpha}  = x_n^jx^{\beta} \in x_n^{j}\widetilde{\mathbf{B}}_{i-j}$, for some $0 \leq j \leq i$ and $x^{\beta} \in \widetilde{\mathbf{B}}_{i-j}$ . Then
$$x^{\alpha}x_n^d  = x_n^{d+j}x^{\beta} \in x_n^{d+j}\widetilde{\mathbf{B}}_{i-j} \subset \mathbf{B}_{d+i}.$$
Thus, the term $c_d.{x^\alpha}x_n^d$ of the product $x^{\alpha}g$ is reduced mod $G$. Therefore, in the coefficients of the reduced form of the product $x^{\alpha}.g$, $c_d$ will appear only in the coefficient of the monomial $x_n^{d+j}x^{\beta} \in \mathbf{\mathbf{B}}_{d+i}$. Thus,
$$\Gamma_{i,d+i} = \pmt{L_{1,1} & \cdots & c_d + L_{1,s} & L_{1,s+1} & \cdots & L_{1,a_{d+i}}  \\ L_{2,1} & \cdots & L_{2,s} & c_d + L_{2,s+1} & \cdots & L_{2,a_{d+i}}  \\ \vdots & \ddots & \vdots & \vdots & \ddots & \vdots \\ L_{a_i,1} & \cdots & L_{a_i,s} & L_{a_i,s+1} & \cdots & c_d + L_{a_i,a_{d+i}}},$$
where $s=a_{d+i} - a_i +1$ and $L_{a,b}$, for $1 \leq a \leq a_i$ and $1 \leq b \leq a_{d+i}$, is a $\bar{F}$-linear combination of the coefficients in $\mathbf{v}_d^*$. Hence, the $a_i$ last columns of $\Gamma_{i,d+i}$ are linearly independent. So,  $\rank(\Gamma_{i,d+i}) = a_i$. This implies that the $a_i$ first linearly independent columns of $M_i$ are the $a_i$ first linearly independent columns of $A_i$ (previously defined). Define $S_i$ to be the subset of $\{1, 2, \cdots, a_{d+i}\}$ such that its elements are the indices of the $a_i$ first linearly independent columns of $A_i$. Then the monomials in $\mathbf{B}_{d+i}^{S_i}$ will be added to $\initial(I)$. Since the $a_{i-1}$ last columns of $\Gamma_{i-1,d+i-1}$ are linearly independent again, we have
$$\rank \pmt{\Gamma_{i,d+i} & \Gamma_{i,d+i-1} \\ 0 & \Gamma_{i-1,d+i-1}}= a_i + a_{i-1}.$$
Therefore, the $a_{i-1}$ next linearly independent columns of $M_i$ are the $a_{i-1}$ first linearly independent columns of $A_{i-1}$ and so on.  We have the $a_i + a_{i-1} + \cdots + a_0$ first linearly independent columns of $M_i$ are the $a_i, a_{i-1}, \cdots, a_1, a_0$ first linearly independent columns of $A_i, A_{i-1}, \cdots, A_1, A_0$ respectively.

However, the monomials in $z\mathbf{B}_{d+i-1}, \cdots, z^{i-1}\mathbf{B}_{d+1}, z^i\mathbf{B}_d$ corresponding to the first linearly independent columns of $A_{i-1}, \cdots, A_1, A_0$ respectively are redundant since they are multiples of the monomials that were already added to $\initial(I)$ in the steps $i-1, \cdots, 1, 0$. Thus, in the step $i$, only the monomials in $\mathbf{B}_{d+i}^{S_i}$ will be added to $\initial(I)$.

Moreover, in the case $\delta - d  = 2k$ we have $\mathbf{B}_{d+k}^{S_k} = B_{d+k}$. Indeed, since $d+k = \delta - k$, by \cite[Lemma 3.4]{JS} one has $a_{d+k} = a_{\delta-k} = a_k$. Hence, $S_k = \{1, 2, \cdots, a_{d+k}\}$.
\vskip 0.2cm
\noindent $\bullet$  For $i^* <  i  < \delta - d$ (in the case $\delta -d \geq 3$), equation (\ref{Equation}) also has form as in (\ref{Equation 1}). Let $\Lambda_i$ denote the square submatrix of $M_i$ given by
$$\mt{& \ \  \mt{\mathbf{B}_{d+i} & z\mathbf{B}_{d+i-1} &  \ \  \cdots & z^{d+2i-\delta}\mathbf{B}_{\delta-i}} \\
\Lambda_i = & \pmt{\Gamma_{i,d+i}&\Gamma_{i,d+i-1}&\cdots&\ \Gamma_{i,\delta-i}\\
0&\Gamma_{i-1,d+i-1}&\cdots&\ \Gamma_{i-1,\delta-i}\\
\vdots&\vdots&\ddots&\ \vdots\\
0&0&\cdots&\ \Gamma_{\delta-d-i,\delta-i}}}.$$
\vskip 0.3cm
Then, $M_i$ has form
$$M_i = \pmt{ \Lambda_i & \Omega \\ 0 & M_{\delta-d-i-1}}.$$
\vskip 0.3cm
By \cite[Proposition 3.16]{JS} $\Lambda_i$ is nonsingular. Hence, the first linearly independent columns of $M_i$ are given by all the columns of $\Lambda_i$ and the columns corresponding to the first linearly independent columns of $M_{\delta-d-i-1}$. Note that the monomials in $E_{d+i}$ corresponding to the first linearly independent columns of $M_{\delta-d-i-1}$ are redundant since they are multiplies of monomials were already added to $\initial(I)$ in the step $\delta-d-i-1$. Furthermore, by using Lemma \ref{multiple}, for $i = i^*+1, i^*+2, \cdots, \delta - d-1$, we obtain the following.
\vskip 0.2cm
If $\delta - d  = 2k$, where $k$ is an integer and $k \geq 2$, then the monomials in $z^2B_{d+k-1}, z^4B_{d+k-2},\\ \cdots, z^{2k-2}B_{d+1}$ will be added to $\initial(I)$.
\vskip 0.2cm
If $\delta - d  = 2k+1$, where $k$ is a positive integer, then the monomials in $B_{d+k+1}, zB_{d+k},\\ z^3B_{d+k-1}, \cdots, z^{2k-1}B_{d+1}$ will be added to $\initial(I)$.
\vskip 0.2cm
\noindent $\bullet$ For $\delta - d \leq i \leq \delta$, by \cite[Corollary 3.11]{JS} the $|E_i|$ first columns of $M_i$ are linearly independent. Hence, we obtain the following.
\vskip 0.2cm
If $\delta - d \geq 2$ then the monomials in $z^{\delta - d}B_d$ , $z^{\delta - d + 2}B_{d-1}$, $\cdots$, $z^{\delta+d-2}B_1$, $z^{\delta+d}B_0$ will be added to $\initial(I)$.
\vskip 0.2cm
If $\delta - d = 1$ then the monomials in $B_{d+1}$, $zB_d$, $z^3B_{d-1}$, $\cdots$, $z^{2d+1}B_0$ will be added to $\initial(I)$.
\end{proof}

\begin{Remark} \emph{The set of generators of $\initial(I,g)$, which appears in Proposition \ref{Structure initial}, is not minimal. For instance the monomial $z^{\delta - d}\mathbf{B}_d^{\{1\}}$ is a multiple of $\mathbf{B}_d^{\{1\}}$.}
\end{Remark}

The following lemma will be useful for proving Proposition \ref{Part 2}.
\begin{Lemma}\label{Lemma} In the case $\delta - d \geq 2$ we have $B_{d+i} \subset \initial(I,g)$ for every $i > i^*$.
\end{Lemma}

\begin{proof}
If $\delta - d  = 2k$, where $k$ is a positive integer, then $i^* =k$. By Proposition \ref{Structure initial}, we have $B_{d+i^*} = \mathbf{B}_{d+i^*}^{S_{i^*}} \subset \initial(I,g)$. Hence, by Lemma \ref{multiple}, $B_{d+i} \subset \initial(I,g)$ for every $i > i^*$.
\vskip 0.2cm
If $\delta - d  = 2k+1$, where $k$ is a positive integer, then $i^* =k$. By Proposition \ref{Structure initial}, we have $B_{d+i^*+1} \subset \initial(I,g)$. Hence, by Lemma \ref{multiple}, $B_{d+i} \subset \initial(I,g)$ for every $i > i^*+1$. Thus, $B_{d+i} \subset \initial(I,g)$ for every $i > i^*$.
\end{proof}
In the next proposition, we will prove that the generators of $\initial(I,g)$ not in $\initial(I)$, $\mathbf{B}_d^{\{1\}}$, $\mathbf{B}_{d+1}^{S_1}$, $\cdots$, $\mathbf{B}_{d+i^*}^{S_{i^*}}$ satisfy property $P$ with respect to $(I,g)$ in any case.
\begin{Proposition}\label{Part 2} {\rm (1)} If $\delta - d = 2k$, where $k$ is a positive integer, then the generators of $\initial(I,g)$ in $z^2B_{d+k-1}$, $z^4B_{d+k-2}$, $\cdots$, $z^{2(d+k)}B_0$ satisfy property $P$ with respect to $(I,g)$.
\vskip 0.2cm
\noindent {\rm (2)} If $\delta - d  = 2k + 1$, where $k$ is a non-negative integer, then the generators of $\initial(I,g)$ in $B_{d+k+1}$, $zB_{d+k}$, $z^3B_{d+k-1}$, $\cdots$, $z^{2(d+k)+1}B_0$ satisfy property $P$ with respect to $(I,g)$.
\end{Proposition}

\begin{proof}
(1) If $x^\mu$ is a generator of $\initial(I,g)$ in $z^2B_{d+k-1}$, $z^4B_{d+k-2}$, $\cdots$, $z^{2(d+k)}B_0$, then $x^\mu$ is divisible by $z$ and $\deg(x^\mu) = l > d+k$. For every monomial $x^\alpha$ of degree $l$ in variables $x_1, \cdots, x_n$, if $x^\alpha \in B_l$, by Lemma \ref{Lemma}, then we have $x^\alpha \in \initial(I,g)$. Otherwise $x^\alpha \notin B_l$, so that $x^\alpha \in \initial(I) \subset \initial(I,g)$. Thus, $x^\mu$ satisfies property $P$ with respect to $(I,g)$.

(2) If $x^\mu$ is a generator of $\initial(I,g)$ in $B_{d+k+1}$, then $x^\mu$ is divisible by $x_n$. Indeed, by \cite[Lemma 3.4]{JS}, $B_{d+k+1} = B_{\delta-k} = x_n^{d+1}B_{k}$. Hence, for every monomial $x^\alpha$ of degree $d+k+1$ in variables $x_1, \cdots, x_{n-1}$, we have $x^\alpha \notin B_{d+k+1}$, so that $x^\alpha \in \initial(I) \subset \initial(I,g)$. Thus, $x^\mu$ satisfies property $P$ with respect to $(I,g)$.

If $x^\mu$ is a generator of $\initial(I,g)$ in $zB_{d+k}$, $z^3B_{d+k-1}$, $\cdots$, $z^{2(d+k)+1}B_0$. Then, by an argument as in (1), $x^\mu$ satisfies property $P$ with respect to $(I,g)$.
\end{proof}
We still have to prove that the minimal generators of $\initial(I,g)$ in $\mathbf{B}_{d+1}^{S_1}, \cdots, \mathbf{B}_{d+i^*}^{S_{i^*}}$ satisfy property $P$ with respect to $(I,g)$. Under condition $d \geq \sigma$, we get the following.
\begin{Proposition} \label{Part 1} If $\sigma \leq d \leq \delta -2$, then the generators of $\initial(I,g)$ in $\mathbf{B}_{d+1}^{S_1}, \cdots, \mathbf{B}_{d+i^*}^{S_{i^*}}$ satisfy  property $P$ with respect to $(I,g)$.
\end{Proposition}

\begin{proof}
Since $d \geq \sigma$ , by Propsition \ref{Structure B} the monomials in $B_{d+1}, \cdots, B_{d+i^*}$ are divisible by $x_n$. Hence, if $x^\mu$ is a generator of $\initial(I,g)$ in $\mathbf{B}_{d+i}^{S_i}$, for some $1 \leq i \leq i^*$,  then  $x^\mu$ is divisible by $x_n$. This implies $x^\mu$ satisfies property $P$ with respect to $(I,g)$ because for every monomial $x^\alpha$ of degree $d+i$ in variables $x_1, \cdots, x_{n-1}$, we have $x^\alpha \notin B_{d+i}$, so that $x^\alpha \in \initial(I) \subset \initial(I,g)$.
\end{proof}

\begin{proof}[Proof of Theorem \ref{Pardue E with condition}.] If $\delta - d  = 2k$, where $k$ is a positive integer, then by Proposition \ref{Structure initial} we have
$$\initial(I,g) = (\initial(I), \mathbf{B}_d^{\{1\}}, \mathbf{B}_{d+1}^{S_1}, \cdots, \mathbf{B}_{d+k}^{S_k}, z^2B_{d+k-1}, z^4B_{d+k-2}, \cdots, z^{2(d+k)}B_0).$$
\noindent The monomial $\mathbf{B}_d^{\{1\}}$ satisfies property $P$ with respect to $(I,g)$ because it is the largest monomial of $\mathbf{B}_d$. By Proposition \ref{Part 2} and Proposition \ref{Part 1} the generators of $\initial(I,g)$ in $z^2B_{d+k-1}, z^4B_{d+k-2}, \cdots, z^{2(d+k)}B_0$ and in $\mathbf{B}_{d+1}^{S_1}, \cdots, \mathbf{B}_{d+k}^{S_k}$   satisfy property $P$ with respect to $(I,g)$. Hence, if Conjecture \ref{Pardue E} is true for $J = \pi(I)$, then every minimal generator of $\initial(I,g)$ satisfies property $P$ with respect to $(I,g)$, so that Conjecture \ref{Pardue E} is true for $\initial(I,g)$.

In case $\delta - d  = 2k + 1$, where $k$ is a non-negative integer, theorem is proved by a completely analogous argument as above.
\end{proof}
From Proposition \ref{Structure initial}, we have the following corollary which describes more explicitly the set of the standard monomials with respect to $(I,g)$ in case $d < \delta$.
\begin{Corollary} \label{Structure F}
Let $(I,g) = (f_1,\cdots,f_n, g)$ be a generic homogeneous ideal of type $(n+1;d_1,\cdots,d_n,d)$ in $K [x_1,\cdots,x_n,z]$, where $d_1 \leq \cdots \leq d_n \leq d$ and $d < \delta$. Let $B = B(n;d_1, \cdots , d_n)$ and $F = B(n+1;d_1,\cdots,d_n,d)$.
\vskip 0.2cm
{\rm (1)} If $\delta - d  = 2k$, where $k$ is a positive integer, then
\vskip 0.2cm
\noindent $F_0 = B_0$,
\vskip 0.2cm
\noindent $F_i = B_i \cup zF_{i-1}$ for every $1 \leq i \leq d-1$,
\vskip 0.2cm
\noindent $F_d = \mathbf{B}_d^{\{2,\cdots,a_d\}} \cup zF_{d-1}$,
\vskip 0.2cm
\noindent $F_{d+i} = \mathbf{B}_{d+i}^{\{1, 2, \cdots, a_{d+i}\}\setminus S_i} \cup zF_{d+i-1}$  for every $1 \leq i \leq k-1$,
\vskip 0.2cm
\noindent $F_{d+k} = zF_{d+k-1}$, $F_{d+k+1} = z^3F_{d+k-2}, \cdots, F_{2(d+k)-1} = z^{2(d+k)-1}F_0.$
\vskip 0.2cm
{\rm (2)} If $\delta - d  = 2k + 1$, where $k$ is a positive integer, then
\vskip 0.2cm
\noindent $F_0 = B_0$,
\vskip 0.2cm
\noindent $F_i = B_i \cup zF_{i-1}$ for every $1 \leq i \leq d-1$,
\vskip 0.2cm
\noindent $F_d = \mathbf{B}_d^{\{2,\cdots,a_d\}} \cup zF_{d-1}$,
\vskip 0.2cm
\noindent $F_{d+i} = \mathbf{B}_{d+i}^{\{1, 2, \cdots, a_{d+i}\}\setminus S_i} \cup zF_{d+i-1}$  for every $1 \leq i \leq k$,
\vskip 0.2cm
\noindent $F_{d+k+1} = z^2F_{d+k-1}$, $F_{d+k+2} = z^4F_{d+k-2}, \cdots, F_{2(d+k)} = z^{2(d+k)}F_0.$
\end{Corollary}
Thus, in order to construct $F = B(n+1, d_1,\cdots,d_n,d)$ from  $B = B(n;d_1,\cdots,d_n)$ we only need to know explicitly the monomials in $\mathbf{B}_{d+1}^{S_1}, \cdots, \mathbf{B}_{d+i^*}^{S_{i^*}}$.

In the following example, we help the reader to construct $\initial(I,g)$ from $\initial(I)$  according to Proposition \ref{Structure initial}. Moreover, we construct $F = B(n+1; d_1, \cdots, d_n,d)$ from $B(n; d_1, \cdots, d_n)$ according to Corollary \ref{Structure F}.
\vskip 0.3cm
\begin{Example}\label{Ex structure in}\emph{Let $(I,g)=(f_1,\cdots,f_4,g)$ be the generic ideal of type $(5; 2,3,3,4,5)$ in $K[x_1, \cdots, x_4, z]$. Let $B = B(4;2,3,3,4)$ as in Example \ref{ex structure B}. Then $\delta = 8, \sigma = 4, d = 5$ and $i^*=\lfloor\frac{\delta-d}{2}\rfloor = 1$. We write  $g$ in reduced form as the following}
$$g = \mathbf{v}_5\mathbf{B}_5 + \mathbf{v}_4\mathbf{B}_4z + \cdots + \mathbf{v}_1\mathbf{B}_1z^4 + z^5.$$

\emph{We will construct a set of generators for $\initial(I,g)$ by using incremental method as in the proof Proposition \ref{Structure initial}. According to equation (\ref{Equation}), $\mathbf{E}_i.g \equiv M_i\mathbf{E}_{i+d} \ \text{(mod} \ G)$, for each $i$ from $0$ to $8$,  we find the monomials that will be added to $\initial(I)$.}

\emph{For $i=0$, the largest monomials of $\mathbf{B}_5$ will be added to $\initial(I)$, in this case it is $x_1x_2x_3^2x_4^2$.}

\emph{For $i=1,$
\begin{align*}
& \quad  \mt{\mathbf{B}_6 & \ z\mathbf{B}_5 & \cdots & z^5\mathbf{B}_1 & z^6} \\
\mathbf{E}_1.g \equiv M_1\mathbf{E}_6 \ \Leftrightarrow \ \pmt{\mathbf{B}_1 \\ z}.g \equiv &\pmt{\Gamma_{1,6} & \Gamma_{1,5} & \cdots & \ \Gamma_{1,1} & \Gamma_{1,0} \\ 0 & \Gamma_{0,5} & \cdots & \ \Gamma_{0,1} & \Gamma_{0,0}},
\end{align*}}

\noindent \emph{where $\mathbf{B}_1 = \mathbf{\widetilde{B}}_1 \cup \{ x_4 \}$ and $\mathbf{B}_6 = x_4^4\mathbf{\widetilde{B}}_2 \cup x_4^5\mathbf{\widetilde{B}}_1 \cup \{x_4^6 \}$. The monomials in $\mathbf{B}_6$ corresponding to the $a_1 = 4$ first linearly independent columns of $\Gamma_{1,6}$ will be added to $\initial(I)$. By using Macaulay2 to compute $\initial(I,g)$, we see that the $4$ largest monomials of $\mathbf{B}_6$ are the minimal generators of $\initial(I,g)$. This means that the $4$ first columns of $\Gamma_{1,6}$ are linearly independent, so that $\mathbf{B}_6^{S_1} = \mathbf{B}_6^{\{1,2,3,4\}}$.}

\emph{For $2 \leq i \leq 8$ the monomials will be added to $\initial(I)$ are $B_7$, $zB_6$, $z^3B_5$, $z^5B_4$, $z^7B_3$, $z^9B_2$, $z^{11}B_1$, $z^{13}$. Thus, the set of generators of $\initial(I,g)$  is
$$\initial(I,g) = (\initial(I), \mathbf{B}_5^{\{1\}}, \mathbf{B}_6^{\{1,2,3,4\}}, B_7, zB_6, z^3B_5, z^5B_4, z^7B_3, z^9B_2, z^{11}B_1, z^{13}).$$ }
\hspace*{10pt}\emph{Let $F = B(5;2,3,3,4,5)$  be the set of the standard monomials with respect to $(I,g)$. Denote by $f_i = |F_i|$ and $f'_i = |\widetilde{F}_i|$. By Corollary \ref{Structure F}, we have}
\emph{\begin{align*}
&F_0 = \{ 1 \}  & f_0 = 1.\\
&F_1 = B_1 \cup \{ z \}   & f'_1 = 4 , f_1 = 5 .\\
&F_2 = B_2 \cup zF_1   & f'_2 = 9 , f_2 = 14 .\\
&F_3 = B_3 \cup zF_2   &  f'_3 = 14, f_3 = 28 .\\
&F_4 = B_4 \cup zF_3   &  f'_4 = 16 , f_4 = 44 .\\
&F_5 = \widetilde{F}_5 \cup zF_4 \ \text{where} \ \widetilde{F}_5 = \mathbf{B}_5^{\{2,3,\cdots,14\}} &  f'_5 = 13, f_5 = 57.\\
&F_6 = \widetilde{F}_6 \cup zF_5 \ \text{where} \ \widetilde{F}_6 = \mathbf{B}_6^{\{5,6,\cdots,9\}} &  f'_6 = 5, f_6 = 62.\\
&F_7 = z^2F_5; \ F_{8} = z^4F_4; \ \cdots; F_{11} = z^{10}F_1; \ F_{12} = z^{12}F_0.&
\end{align*}}
\end{Example}
In \cite[Conjecture 3.14]{JS}, it is conjectured that $\mathbf{B}_{d+i}^{S_i}$ are the $a_i$ largest monomials of $\mathbf{B}_{d+i}$ for every $i = 0, \cdots, i^*$. However, in the following example, we show that this conjecture is not true.
\vskip 0.3cm
\begin{Example} \emph{Let $(I,g)=(f_1,\cdots,f_5,g)$ be the generic ideal of type $(6; 2,3,3,4,5,5)$ in $K[x_1, \cdots, x_5, z]$. Let $F = B(5;2,3,3,4,5)$ as in Example \ref{Ex structure in} with the variable $x_5$ instead of variable $z$. Then $\delta = 12, \sigma = 6$, $d=5$ and $i^*=\lfloor\frac{\delta-d}{2}\rfloor = 3$. Here $F$ plays the same  role of $B$ in Proposition \ref{Structure initial}. We write $g$ in reduced form as the following}
$$g = \mathbf{v}_5\mathbf{F}_5 + \mathbf{v}_4\mathbf{F}_4z + \cdots + \mathbf{v}_1\mathbf{F}_1z^4 + z^5.$$

\emph{According to the incremental method, for $i=0$, the largest monomials of $\mathbf{F}_5$ will be added to $\initial(I)$, in this case it is $x_2^2x_3x_4^2$.}

\emph{For $i=1,$
\begin{align*}
& \quad  \mt{\mathbf{F}_6 & \ z\mathbf{F}_5 & \cdots & z^5\mathbf{F}_1 & z^6} \\
\mathbf{E}_1.g \equiv M_1\mathbf{E}_6 \ \Leftrightarrow \ \pmt{\mathbf{F}_1 \\ z}.g \equiv &\pmt{\Gamma_{1,6} & \Gamma_{1,5} & \cdots & \ \Gamma_{1,1} & \Gamma_{1,0} \\ 0 & \Gamma_{0,5} & \cdots & \ \Gamma_{0,1} & \Gamma_{0,0}},
\end{align*}}

\noindent \emph{where $\mathbf{F}_1 = \mathbf{\widetilde{F}}_1 \cup \{ x_5 \}$ and $\mathbf{F}_6 = \mathbf{\widetilde{F}}_6 \cup x_5\mathbf{\widetilde{F}}_5 \cdots \cup x_5^5\mathbf{\widetilde{F}}_1 \cup \{x_5^6 \}$. The monomials in $\mathbf{F}_6$ corresponding to the $f_1=5$ first linearly independent columns of $\Gamma_{1,6}$ will be added to $\initial(I)$. We have}
\begin{align*}
& \quad  \mt{\mathbf{\widetilde{F}}_6 & x_5\mathbf{\widetilde{F}}_5 & \cdots & x_5^5\mathbf{\widetilde{F}}_1 & x_5^6 } \\
\mathbf{F}_1.g \longleftrightarrow \Gamma_{1,6}\mathbf{F}_6 \ \Leftrightarrow \ \pmt{\mathbf{\widetilde{F}}_1 \\ x_5}.g \longleftrightarrow &\pmt{\Omega_{1,6} & \Omega_{1,5} & \cdots & \ \Omega_{1,1} & \Omega_{1,0} \\ 0 & \Omega_{0,5} & \cdots & \ \Omega_{0,1} & \Omega_{0,0}},
\end{align*}
\emph{Since $|\mathbf{\widetilde{F}}_1| = f'_1 = 4$ and $|\mathbf{\widetilde{F}}_6| = f'_6 = 5$, we get $\rank(\Omega_{1,6}) \leq 4$. By using Macaulay2 to compute $\initial(I,g)$, we see that the $4$ largest monomials of $\mathbf{F}_6$ are the minimal generators of $\initial(I,g)$. This means that the $4$ first columns of $\Omega_{1,6}$ are linearly independent. Hence the $5$ first linearly independent columns of $\Gamma_{1,6}$ are the $4$ first columns and the $6$-th column. Thus $\mathbf{F}_6^{S_1} = \mathbf{F}_6^{\{1,2,3,4\}} \cup \mathbf{F}_6^{\{6\}}$. However the monomial $\mathbf{F}_6^{\{6\}}$ is not a minimal generator because $\mathbf{F}_6^{\{6\}} = x_5\mathbf{F}_5^{\{1\}}$ and $\mathbf{F}_5^{\{1\}}$ was already added to $\initial(I)$ in step $i=0$.}

\end{Example}

\section{Application for Fr\"{o}berg's Conjecture}
In \cite[Theorem 2]{P}, Pardue proved that Fr\"{o}berg's Conjecture is equivalent to Conjecture \ref{Pardue E}. In order to prove the equivalence of the conjectures, Pardue used the notion of semi-regular sequences that was introduced in \cite[Section 3]{P}.
Regular sequences and semi-regular sequences can be characterized by Hilbert series.

\begin{Proposition} \cite[Proposition 1]{P}\label{semi-regular}
Let $A = K[x_1,\cdots,x_n]/I$, where $I$ is a homogeneous ideal, and $f_1, \cdots, f_r$ are homogeneous polynomials of degree $d_1, \cdots, d_r$. Then,
\vskip 0.3cm
\noindent {\rm (1)} $f_1, \cdots, f_r$ is a semi-regular sequence on $A$ if and only if for all $s = 1, \cdots, r$
$$HS_{A/(f_1, \cdots, f_s)}(z) = \left\lceil \pmt{\prod_{i=1}^s(1 - z^{d_i})} HS_A(z) \right\rceil.$$
{\rm (2)} $f_1, \cdots, f_r$ is a regular sequence on $A$ if and only if
$$HS_{A/(f_1, \cdots, f_r)}(z) = \pmt{\prod_{i=1}^r(1 - z^{d_i})} HS_A(z).$$
\end{Proposition}

\vskip 0.3cm
In \cite[Theorem 2]{P}, Pardue also proved that  Conjecture \ref{Pardue E} is equivalent to the following conjecture.

\begin{Conjecture} \cite[Conjecture C]{P}\label{Pardue C}
Let $I = (f_1, \cdots, f_n)$ be a generic homogeneous ideal of type $(n;d_1,\cdots,d_n)$ in $K [x_1,\cdots,x_n]$. Then $x_n, x_{n-1}, \cdots, x_1$ is a semi-regular sequence on $A = K[x_1, \cdots, x_n]/I$.
\end{Conjecture}

We apply now Theorem \ref{Pardue E with condition} to get partial answers to Conjecture \ref{Pardue E} and Conjecture \ref{Pardue C}. Let $d_1 \leq \cdots \leq d_n$ be $n$ positive integers. For every $1 \leq i \leq n$, we set
\vskip 0.3cm
\hspace{120pt} $\delta_i = d_1 + \cdots + d_{i} - i,$
\vskip 0.3cm
\hspace{120pt} $\sigma_i = \min \left\{ \delta_{i-1} , \left\lfloor \dfrac{\delta_i}{2} \right\rfloor \right\}$ for all $i \geq 2$.
\vskip 0.3cm
\begin{Theorem}\label{Partial Pardue E3}
Let $I = (f_1,\cdots,f_n)$ be a generic homogeneous ideal of type $(n;d_1,\cdots,d_n)$ in $K [x_1,\cdots,x_n]$ with $n \leq 3$ and $d_1 \leq \cdots \leq d_n$. Then, Conjecture \ref{Pardue E} is true for $I$.
\end{Theorem}
\begin{proof}
It is known that Conjecture \ref{Pardue E} is true in case $n \leq 2$. For $n=3$, we have $J = \pi(I)$ is a generic ideal of type $(2;d_1,d_2)$. Hence, Conjecture \ref{Pardue E} is true for $J$. Since

$$d_3 \geq \sigma_2 = \min \left\{ d_1 - 1 , \left\lfloor \dfrac{d_1 + d_2 - 2}{2} \right\rfloor \right\},$$
\vskip 0.3cm
\noindent by Theorem \ref{Pardue E with condition} we have that Conjecture \ref{Pardue E} is true for $I$.
\end{proof}
\begin{Theorem}\label{Partial Pardue E}
Let $I = (f_1,\cdots,f_n)$ be a generic homogeneous ideal of type $(n;d_1,\cdots,d_n)$ in $K [x_1,\cdots,x_n]$ with $n \geq 4$ and $d_1 \leq \cdots \leq d_n$. If $d_i \geq \sigma_{i-1}$ for all $4 \leq i \leq n$, then Conjecture \ref{Pardue E} is true for $I$.
\end{Theorem}
\begin{proof}
We prove by induction on $n$. For $n=4$, we have $J = \pi(f_1,f_2,f_3)$ is a generic ideal of type $(3;d_1,d_2,d_3)$. By Theorem \ref{Partial Pardue E3}, Conjecture \ref{Pardue E} is true for $J$. Since $d_4 \geq \sigma_3$, by Theorem \ref{Pardue E with condition} we have that Conjecture \ref{Pardue E} is true for $I$.

For $n > 4$, we have $J = \pi(f_1,\cdots,f_{n-1})$ is a generic ideal of type $(n-1;d_1,\cdots,d_{n-1})$ with $d_i \geq \sigma_{i-1}$ for all $4 \leq i \leq n-1$. Hence, by induction Conjecture \ref{Pardue E} is true for $J$. Since, $d_n \geq \sigma_{n-1}$, by Theorem \ref{Pardue E with condition} we have that Conjecture \ref{Pardue E} is true for $I$.
\end{proof}

Since Conjecture \ref{Pardue E} is equivalent to Conjecture \ref{Pardue C}, we also get a partial answer to Conjecture \ref{Pardue C}.

\begin{Corollary}\label{Partial Pardue C}
 Let $I = (f_1,\cdots,f_n)$ be a generic homogeneous ideal of type $(n;d_1,\cdots,d_n)$ in $K [x_1,\cdots,x_n]$ with $d_1 \leq \cdots \leq d_n$. If $n \leq 3$ or $n \geq 4$ and $d_i \geq \sigma_{i-1}$ for all $4 \leq i \leq n$, then $x_n, x_{n-1}, \cdots, x_1$ is a semi-regular sequence on $K[x_1, \cdots, x_n]/I$.
\end{Corollary}

We apply now above results to prove a new partial answer for Fr\"{o}berg's Conjecture.

\begin{Theorem}\label{Partial FB}
Let $I = (f_1,\cdots,f_r)$ be a generic homogeneous ideal of type $(n;d_1,\cdots,d_r)$ in $R = K [x_1,\cdots,x_n]$ with $r \leq n+2$ and $d_1 \leq \cdots \leq d_r$. If $r \leq 3$ or $r \geq 4$ and $d_i \geq \sigma_{i-1}$ for all $4 \leq i \leq r$, then the Hilbert series of $R/I$ is given by
$$HS_{R/I}(z) =  \left\lceil \dfrac{\prod_{i=1}^r(1 - z^{d_i})}{(1-z)^n} \right\rceil.$$
\end{Theorem}

\begin{proof}
Since Fr\"{o}berg's Conjecture is known to be true if $r \leq n$, we only have to consider the case  $r > n$. Let $R' = K[x_1, \cdots, x_r]$ be the polynomial ring in $r$ variables and view $R$ as $R = R'/(x_r, \cdots, x_{n+1})$. Then, there exist the generic homogeneous polynomials $f_1', \cdots, f_r'$ of type $(r; d_1, \cdots, d_r)$ in $R'$ such that $f_i$ is the image of $f_i'$ in $R = R'/(x_r, \cdots, x_{n+1})$. Set $A = R'/(f_1', \cdots, f_r')$. It is known that $A$ is the complete intersection and Hilbert series of $A$ is given by
$$HS_A(z) =  \dfrac{\prod_{i=1}^r(1 - z^{d_i})}{(1-z)^r}.$$

Applying Corollary \ref{Partial Pardue C} for $(f_1', \cdots, f_r')$ we have $x_r, \cdots, x_{n+1}, \cdots, x_1$ is a semi-regular sequence on $A$. By Proposition \ref{semi-regular} we get
$$HS_{A/(x_r, \cdots, x_{n+1})}(z) =  \left\lceil (1-z)^{r-n}HS_A(z) \right\rceil = \left\lceil \dfrac{\prod_{i=1}^r(1 - z^{d_i})}{(1-z)^n} \right\rceil.$$
and the theorem follows from the following isomorphisms.
$$A/(x_r, \cdots, x_{n+1}) \cong R'/(f_1', \cdots, f_r', x_r, \cdots, x_{n+1}) \cong R/(f_1, \cdots, f_r).$$
\end{proof}

\begin{Remark}\label{Remark Lisa} \emph{Let $I = (f_1,\cdots,f_{n+1})$ be a generic ideal of type $(n;d_1,\cdots,d_{n+1})$ in $R=K[x_1,\cdots,x_n]$ with $d_1 \leq d_2 \leq \cdots \leq d_{n+1}$. If the Hilbert series of $R/I$ is given by
$$HS_{R/I}(z) =  \left\lceil \dfrac{\prod_{i=1}^{n+1}(1 - z^{d_i})}{(1-z)^n} \right\rceil,$$
then $\sigma_{n+1}$ is the largest number such that $(R/I)_t$ is non-zero for every $t \leq \sigma_{n+1}$. Hence, in Theorem \ref{Partial FB}, the degree of $f_{n+2}$ should be equal to $\sigma_{n+1}.$  }
\end{Remark}

\noindent \textbf{Acknowledgements}
I thank my advisor Maria Evelina Rossi for suggesting the problem and for providing helpful suggestions throughout the preparation of this manuscript.    I am also grateful to the department of Mathematics of Genova University for supporting my PhD program.   I also thank Lisa Nicklasson for her helpful comments  which brought me to clarify some issues on Theorem \ref{Partial FB}  and contained in  Remark \ref{Remark Lisa}. I would like to thank the referees for their comments which have improved the presentation of this final version.


\end{document}